\documentclass[12pt]{article}
\usepackage{textcomp}
\usepackage{url}
\usepackage{amsthm}
\usepackage{amsmath}
\usepackage{fullpage}
\usepackage{enumerate}
\usepackage{amssymb}

\usepackage{booktabs}
\usepackage[boxed]{algorithm2e}
\usepackage{float}

\newtheorem{thm}{Theorem}
\newtheorem{lem}{Lemma}
\newtheorem{cor}{Corollary}

\title{Lehmer numbers and  primitive roots modulo a prime}
\author{Stephen D. Cohen \footnote {Postal address: 6 Bracken Road, Portlethen, Aberdeen AB12 4TA, Scotland.}\\
{\small University of Glasgow, Scotland}\\
  {\small Stephen.Cohen@glasgow.ac.uk}
\and
  Tim Trudgian\footnote{Supported by Australian Research Council Future Fellowship FT160100094.}\\
{\small School of Physical, Environmental and Mathematical Sciences}\\{\small  The University of New South Wales Canberra, Australia }\\
{\small  t.trudgian@adfa.edu.au } }

\begin{document}
\maketitle
\begin{abstract}
\noindent
A Lehmer number 
modulo a prime $p$  is an integer $a$ with $1 \leq a \leq p-1$ whose inverse $\bar{a}$ within
 the same range has opposite parity. Lehmer numbers that are also primitive roots have been discussed by Wang and Wang  \cite{WW} in an endeavour to count the number
 of ways  $1$ can be expressed as  the sum of two primitive roots that are also Lehmer numbers
 (an extension of a question of S.\ Golomb).  
 In this paper we give an explicit estimate for the number of Lehmer  primitive roots modulo $p$ and prove that, for all primes $p \neq 2,3,7$, Lehmer primitive roots exist.  We also make
  explicit the  known  expression for the number of Lehmer numbers modulo $p$ and improve the Wang--Wang estimate
   for the number of solutions to the Golomb--Lehmer primitive root problem.
  \end{abstract}

\section{Introduction}

Let $p$ be an odd prime and $a$ an integer with $1 \leq a \leq p-1$.   Define $\bar{a} $ to be integer with $1 \leq \bar{a} \leq p-1$
such that $\bar{a}$ is the inverse of $a$ modulo $p$.  Following the interest in such integers by D. H. Lehmer (see, e.g.\ \cite[\S F12]{Guy}) we define $a$ to be a {\em Lehmer number}
if $a$ and $\bar{a}$ have opposite parity, i.e., $a+\bar{a}$ is odd.  Thus $a$ is a Lehmer number if and only if $\bar{a}$ is a Lehmer number.
 It is easily checked that there are no Lehmer numbers modulo $p$ when $p=3$  or $7$.

W. Zhang  \cite{zhang94}  has shown
that $M_p$, the number of Lehmer numbers modulo $p$, satisfies
\begin{equation}\label{zhang}
M_p =\frac{p-1}{2} + O(p^{\frac{1}{2}} \log^2 p).
\end{equation}
We make this explicit in Theorem \ref{straw} below.

 A Lehmer number which is also a primitive root modulo $p$ will be called a {\em Lehmer primitive root} or an {\em LPR}.  The inverse $\bar{a}$ of an LPR is also an LPR.
  Since there is no Lehmer number  modulo $3$,  we can suppose $p>3$.
 Wang and Wang  \cite{WW} consider LPRs in an analogue of the question
of Golomb relating to pairs $(a,b)$ of primitive roots modulo $p$ for which $a+b \equiv 1 \pmod p$.  Specifically, Wang and Wang
derive  an asymptotic estimate for $G_p$, the  number of pairs $(a,b)$ of LPRs  for which  $a+b \equiv 1 \pmod p$ (thus $a+b=p+1$), namely,
\begin{equation} \label{wangwang}
G_p = \theta_{p-1}^2\left(\frac{p-1}{4} +O(W_{p-1}^2 p^{\frac{3}{4}}\log^2 p)\right),
\end{equation}
where, for a positive integer $m$, $\theta_m = \frac{\phi(m)}{m}$ ($\phi$ being Euler's function)
and $W_m = 2^{\omega(m)}$ is the number of square-free divisors of $m$.  It follows from  (\ref{wangwang}) that
there is always a pair $(a,b)$ of LPRs modulo $p$ for which $a+b=p+1$ for sufficiently large $p$. Since the result is inexplicit it is an open problem to specify which primes $p$ (if any) fail to possess such a pair $(a,b)$. 

As a preliminary it is clearly desirable to possess an asymptotic expression analogous to (\ref{zhang}) and (\ref{wangwang})
for $N_p$ defined simply as the number of LPRs modulo a prime $p \ (>3)$ and also to exhibit explicitly the finite list of primes $p$ for which there exists no LPR modulo $p$.
This is the main purpose of the present article.

For odd integers $m\geq3$ define the positive number $T_m$ by
\begin{equation} \label{Tm}
T_m=\frac{2\sum_{j=1}^{(m-1)/2}\tan\left(\frac{\pi j}{m}\right)}{m\log m}.
\end{equation}
The asymptotic result to be proved is the following.

\begin{thm} \label{Np}
Let $p>3$ be a prime.  Then
\begin{equation}\label{Npeq}
\left|N_{p}- \frac{\phi(p-1)}{2}\right| < T_p^2\theta_{p-1} W_{p-1}p^{\frac{1}{2}}\log^2 p .
\end{equation}
In particular, if $p >3$, then
\begin{equation}\label{Npeq1}
\left|N_{p}- \frac{\phi(p-1)}{2}\right| <\frac{1}{2}\theta_{p-1}W_{p-1}p^{\frac{1}{2}} \log^2 p.
\end{equation}
\end{thm}

A criterion for the existence of an LPR follows immediately from Theorem \ref{Np}.
\begin{cor} \label{exist}
Let $p>3$ be a prime.   Suppose that
\[p^{\frac{1}{2}}>2T_p^2W_{p-1} \log^2 p + p^{-\frac{1}{2}}.\]  Then there exists an LPR modulo $p$.  In particular, provided $p > 7$,  it suffices that
\begin{equation}\label{dundee}
p^{\frac{1}{2}}> W_{p-1}\log^2 p + p^{-\frac{1}{2}}.
\end{equation}
\end{cor}
In fact, a complete  existence result will be proved as follows.
\begin{thm}\label{lpr}
Suppose $p (\neq3, 7)$ is an odd prime.  Then there exists an {\em LPR} modulo $p$.
\end{thm}

Finally we obtain an  improvement to (\ref{wangwang}), namely,
\begin{equation} \label{indigo}
\left|G_p- \frac{\theta_{p-1}^2}{4}(p-2)\right| <\frac{\theta_{p-1}^2}{8}[W_{p-1}^2(9\log^2 p+1) -1]p^{\frac{1}{2}}, \quad p> 3.
\end{equation}
Of course, (\ref{indigo}) implies that. for sufficiently large primes $p$, there exists a pair $(a,b)$  of LPR modulo $p$ such that $a+b\equiv 1  \pmod p$.  We defer a full discussion 
of the existence question, however, to a future investigation.  

The outline of this paper is as follows. In \S \ref{Twain} we give bounds for the function $T_{m}$ introduced in (\ref{Tm}). In \S \ref{Thackeray} we prove Theorem \ref{straw}, which is an explicit version of (\ref{zhang}). In \S\ref{Forster} we prove Theorem \ref{Np} and introduce a sieve. This enables us to prove Theorem \ref{lpr} in \S\ref{Conrad}. Finally, in \S \ref{Macauley} we prove (\ref{indigo}) in Theorem \ref{cream} thereby improving on the main result from Wang and Wang \cite{WW}. 

The authors are grateful to Maike Massierer who provided much useful advice relating to the computations in \S \ref{Conrad}.

\section{Bounds for  $T_m$}\label{Twain}

The sum $T_m$  is relevant to previous work on Lehmer numbers  (such as \cite{zhang94} and \cite{WW}).
 For explicit results it is helpful to have better bounds than those used in these papers.  Here, Lemma \ref{red} below (while not best possible) is sufficient for
 our purposes.  Indeed, only the upper bound is needed in what follows.    We remark that $1+\log(\frac{2}{\pi})=0.54841\ldots$.

\begin{lem}\label{red}
 For any odd integer $m \geq 3$ we have
 \begin{equation}\label{tartan}
\frac{2}{\pi}\left(1+ \frac{0.548}{\log {m}}\right) <T_m< \frac{2}{\pi}\left(1+ \frac{1.549}{\log {m}}\right).
\end{equation}
In particular, if $m\geq 1637$, then  $T_m^2< \frac{1}{2}$.
 \end{lem}

 \begin{proof}  We begin with the upper bound for $T_m$ in (\ref{tartan}). Since $\tan x$ is an increasing function for $0 \leq x < \pi/2$, then
  $$S_m=\sum_{j=0}^{(m-3)/2} \tan  \left(\frac{\pi j}{m}\right)= \sum_{j=1}^{(m-3)/2} \tan\left(\frac{\pi j}{m}\right)$$
 is a left-Riemmann sum  (with unit intervals) for the integral $\int_{0}^{(m-1)/2}  \tan\left(\frac{\pi x}{2}\right) dx$, so that
 \[ S_m < \frac{m}{\pi}\log \sec\left(\frac{\pi(m-1)}{2m}\right)= \frac{m}{\pi}\log \csc\left(\frac{\pi }{2m}\right)\]
 Hence
 \[ T_m m \log m<\frac{2m}{\pi}\log \csc\left(\frac{\pi }{2m}\right)+ 2\tan\left(\frac{\pi (m-1)}{2m}\right)<\frac{2m}{\pi}\log \csc\left(\frac{\pi }{2m}\right)+2 \csc\left(\frac{\pi}{2m}\right)\]
 Now $\sin x > x-x^3/6$, whence, with $\beta= \pi^2/(24m^2)$,
 $$\csc \left(\frac{\pi}{2m}\right)<\frac{2m}{\pi}(1-\beta)^{-1}< \frac{2m}{\pi}(1+2 \beta), $$
since, certainly,  $\beta< 1/2$.  It follows that
\[T_m m\log m < \frac{2m}{\pi}(\log m + \log(2/\pi)+ 2\beta+ 2+ 4 \beta)<\frac{2m}{\pi}(\log m + 1.549)\]
provided $m >101$.  The first claimed inequality follows for $m\geq 101$.  In fact, by calculation it is also true for all smaller values of $m$.

 From this, if $m>1200001$,  we have $T_m< 0.7071$ and hence $T_m^2< 1/2$.  By direct calculation, this
inequality also holds for $1637 \leq m <1200001$.

For the left hand inequality of (\ref{tartan}), we exploit the fact that $S_m+\tan\left(\frac{\pi (m-1)}{2m}\right)$ is the trapezoidal rule approximation to the integral
$\int_{0}^{(m-1)/2} 2 \tan\left(\frac{\pi x}{2}\right) dx$.  Indeed, since the integrand is concave up, the error term (involving the second derivative) is negative, i.e., the
sum exceeds the integral.  Hence
\[T_m m \log m>  \frac{2m}{\pi}\log \csc\left(\frac{\pi }{2m}\right)+ \tan\left(\frac{\pi (m-1)}{2m}\right)> \frac{2m}{\pi}\log\frac{2m}{\pi}+\cot\frac{\pi}{2m}.\]
For $0<x<1, \cos x >1-x^2/2$ and $\sin x <x$ so that  $\cot\frac{\pi}{2m}>\frac{2m}{\pi}\left(1-\frac{\pi^2}{8m^2}\right)>\frac{2m}{\pi}(1-0.0001)$ whenever $m\geq111$.
Moreover, $ \csc\left(\frac{\pi }{2m}\right) >\frac{2m}{\pi}$, whence, whenever $m \geq 111$,
\[T_m>\frac{2}{\pi}\left(1+\frac{\frac{2}{\pi}+1-0.0001}{\log m}\right) \]
The result follows for $m\geq 111$. It also holds when $3\leq m<111$ by direct calculation.
\end{proof}
\section{The number of Lehmer numbers modulo $p$}\label{Thackeray}
We turn to making (\ref{zhang}) explicit.  For this we acknowledge the ideas of \cite{zhang94} and \cite{WW}.
\begin{thm}\label{straw}
Suppose $p>3$ is a prime. Then
\begin{equation} \label{zhang2}
  \left|M_p -\frac{p-1}{2}\right| <T_p^2 p^{\frac{1}{2}} \log^2 p.\end{equation} Moreover, for all $p$ we have
\begin{equation}\label{chips}
\left|M_p -\frac{p-1}{2}\right| <\frac{1}{2}p^{\frac{1}{2}} \log^2 p.
\end{equation}
\end{thm}
\begin{proof}  Evidently,
\begin{equation}\label{lime}
M_p = \frac{1}{2} \sum_{a=1}^{p-1}(1-(-1)^{a+\bar{a}})=\frac{p-1}{2} -\sum_{a=1}^{p-1}(-1)^{a+\bar{a}}=\frac{p-1}{2} -\frac{1}{2}E_p,
\end{equation}
say. Let $\psi$ be the additive character on the integers modulo $p$ defined by $\psi(a)=\exp(2\pi i a/p)$. Express the function $(-1)^a$ in terms of additive characters modulo $p$ using
the transformation
\[(-1)^a =\frac{1}{p}\sum_{r=1}^{p-1}\sum_{j=0}^{p-1}(-1)^r\psi(j(a-r))=\frac{1}{p}\sum_{r=1}^{p-1}(-1)^r\sum_{j=0}^{p-1}\psi(j(a-r)).\]
Similarly,
\[(-1)^{\bar{a}} =\frac{1}{p}\sum_{s=1}^{p-1}(-1)^s\sum_{k=0}^{p-1}\psi(k(\bar{a}-s)).\]
Hence,
\[E_p= \frac{1}{p^2}\sum_{j,k=0}^{p-1}\sum_{a=1}^{p-1}\psi(ja+k\bar{a})
\sum_{r=1}^{p-1}(-1)^r\psi(-jr)\sum_{s=1}^{p-1}(-1)^s \psi(-ks).\]
Notice that, if $j=0$, then $\sum_{r=1}^{p-1}(-1)^r\psi(-jr) =\sum_{r=1}^{p-1}(-1)^r=0$, since $p$ is odd.   Hence, we can suppose the range of $j$ and, similarly, of $k$ in $E_p$ runs from
$1$ to $p-1$.  Thus
\begin{equation}\label{pink}
|E_p|=\frac{1}{p^2}\sum_{j,k=1}^{p-1}\big{|}\sum_{a=1}^{p-1}\psi(ja+k\bar{a})\big{|} \left|\sum_{r=1}^{p-1}(-1)^r\psi(-jr)\right|
\left|\sum_{s=1}^{p-1}(-1)^s \psi(-ks)\right|.
\end{equation}
Now $\sum_{a=1}^{p-1}\psi(ja+k\bar{a})$ is a Kloosterman sum and so is bounded by $2p^{\frac{1}{2}}$, whatever the values of $j,k$.

Next, in (\ref{pink}),

\[\sum_{r=1}^{p-1}(-1)^r\psi(-jr)=\frac{1-\exp(2\pi ij/p)}{1+\exp(2\pi ij/p)}=\frac {i \sin (\pi j/p)}{\cos (\pi j/p)}.\]
Moreover,
\[\sum_{j=1}^{(p-1)}\left|\frac { \sin (\pi j/p)}{\cos (\pi j/p)}\right|=2\sum_{j=1}^{(p-1)/2} \tan \left(\frac{\pi j}{p}\right)=T_p\ p \log p, \]
by the definition (\ref{Tm}).  It follows that
\[\left|\sum_{r=1}^{p-1}(-1)^r\psi(-jr)\right| < T_p \  p \log p,\]
 and, similarly, for the sum in (\ref{pink}) over $k$.

 Applying these bounds to (\ref{pink}), from (\ref{lime}) we deduce (\ref{zhang2}). Using Lemma 1 and a small computation we deduce (\ref{chips}).
\end{proof}

\section{A slight extension of Theorem \ref{Np} and its proof}\label{Forster}

Throughout let $p >3$ be a prime. All references given will be modulo $p$ (unless otherwise mentioned).
We begin by extending the concept of a primitive root (as used in a number of papers such as \cite{COT}).  For any {\em even} divisor $e$ of $p-1$
an integer $a$ (indivisible by $p$) will be said to be $e$-free if $a \equiv b^d \pmod p$ for an integer $b$ and divisor $d$ of $e$ implies
$d=1$.  Thus $a$ is a primitive root if it is $p-1$-free,   Indeed, $a$ is a primitive root if and only if $a$ is $l$-free for all prime divisors $l$ of $p-1$.
More generally, $a$ is $e$-free if and only if it is $l$-free for all prime divisors $l$ of $e$. It follows that the proportion of integers in $[1,p-1]$ which are $e$-free
is $\theta_e$ and therefore that their total number is $\theta_e (p-1)$.

Now, the  function
\begin{equation*}\label{efree}
\theta_e\sum_{d|e}\frac{\mu(d)}{\phi(d)} \sum_{\chi_d}\chi_d
\end{equation*}
acting on integers $a$ (indivisible by $p$) takes the value $1$ if $a$ is $e$-free and is zero, otherwise.  Here the sum over $\chi_d$ is  over all $\phi(d)$
multiplicative characters $\chi_d$ modulo $p$ of  order $d$.

The criterion for an integer $a$ with $1 \leq a \leq p-1$ to be a Lehmer number is that $\frac{1}{2}(1- (-1)^{a+\bar{a}})=1$ (and not 0).  For any divisor $e$ of $p-1$,
write $N_p(e)=N(e)$ for the number of Lehmer numbers $a$  such that $a$ is also $e$-free.  In particular, $N(p-1)=N_p$ is the number of LPRs modulo $p$.  By the above,
\begin{equation*}\label{Neexpr}
N(e) = \frac{1}{2} \theta_e \sum_{d|e} \frac{\mu(d)}{\phi(d)} \sum_{\chi_d}\sum_{1 \leq a \leq p-1}(1-(-1)^{a+\bar{a}})\chi_d(a).
\end{equation*}
In fact the sum $\theta_e \sum_{d|e} \frac{\mu(d)}{\phi(d)} \sum_{\chi_d}\sum_{1 \leq a \leq p-1}\chi_d(a)$ simply yields the number of $e$-free integers modulo $p$, namely
$\theta_e (p-1)$.  Hence
\begin{equation}\label{white}
N(e)=\frac{\theta_e}{2} (p-1)-\frac{1}{2}E(e),
\end{equation}
where
\begin{equation}\label{grey}
E(e) = \theta_e \sum_{d|e} \frac{\mu(d)}{\phi(d)} \sum_{\chi_d}\sum_{a=1}^{p-1}(-1)^{a+\bar{a}}\chi_d(a).
\end{equation}

As for (\ref{pink}) we obtain
\
\begin{equation}\label{black}
|E(e)|=\frac{\theta_e}{p^2}\sum_{d|e}\frac{|\mu(d)|}{\phi(d)}\sum_{\chi_d}\sum_{j,k=1}^{p-1}\big{|}\sum_{a=1}^{p-1}\chi_d(a)\psi(ja+k\bar{a})\big{|} \left|\sum_{r=1}^{p-1}(-1)^r\psi(-jr)\right|
\left|\sum_{s=1}^{p-1}(-1)^s \psi(-ks)\right|.
\end{equation}

Now, regarding $(ja+k\bar{a})$ in (\ref{black}) as the rational function $(ja^2+k)/a$, we have, by a theorem of Castro and Moreno (see (1.4) of \cite{CP}), that,
for each pair $(j,k)$ with $1 \leq j,k \leq p-1$,
\begin{equation} \label{blue}
\big{|}\sum_{a=1}^{p-1}\chi_d(a)\psi(ja+k\bar{a})\big{|} \leq 2p^{\frac{1}{2}},
\end{equation}
a bound which is independent of $j$ and $k$.

As we have already seen
\begin{equation} \label{Tp}
\left|\sum_{r=1}^{p-1}(-1)^r\psi(-jr)\right | < T_p \  p \log p,
\end{equation}
 and, similarly, for the sum in (\ref{black}) over $k$.

Since there are $\phi(d)$ characters $\chi_d$ of degree $d$ and $\sum_{d|e}|\mu(d)|=W_e$, we deduce from (\ref{black}) by means of the bounds (\ref{blue})
and (\ref{Tp}) that
\begin{equation} \label{green}
 |E(e)| <2\theta_eW_eT_p^2 \ p^{\frac{1}{2}}\log^2 p.
\end{equation}

Hence  (\ref{Npeq}) is immediate from (\ref{green}) with $e=p-1$ and (\ref{Npeq1}) follows by
    Lemma \ref{red}.  More generally, by means of Lemma \ref{red}, we have established the following extension of Theorem \ref{Np}.

  \begin{thm} \label{Ne}
Let $p>3$ be a prime and $e$ an even divisor of $p-1$.  Then
\begin{equation}\label{Neeq}
\left|N_p(e)- \frac{\theta_e}{2}(p-1)\right| < T_p^2\theta_e W_ep^{\frac{1}{2}}\log^2 p.
\end{equation}
\end{thm}
The estimate (\ref{Npeq}) of Theorem \ref{Np} follows from Theorem \ref{Ne}  by selecting $e=p-1$.  We deduce (\ref{Npeq1}) by Lemma \ref{red} for $p \geq 1637$ and then for smaller prime values by simple direct calculation.

\section {Proof of the existence theorem}\label{Conrad}
We shall use Theorem \ref{Np} to obtain an  existence result for (explicitly) large primes $p$.  In order to extend the range of the method, however, we first describe a ``sieving"
approach based on Theorem \ref{Ne} similar to that used in \cite{COT} and many other papers associated with the authors.

Set $\omega=\omega(p-1)$.  Let $f$ be an even divisor of $p-1$ which is the product of the $r(\geq 1)$ smallest distinct prime factors of $p-1$ ($f$ is the {\em core}).
 Further  let the remaining distinct prime factors of $p-1$ be $p_1, \ldots, p_s$ (the sieving primes). Define $\delta= 1 - \sum_{i=1}^s\frac{1}{p_i}$. As in previous work on related problems (\cite{COT15} and \cite{COT}) we have the following.

 \begin{lem}\label{sieve} With the above notation,
 \[N_p \geq \sum_{i=1}^sN(p_if)- (s-1)N(f).\]
 Hence
 \begin{equation}\label{old}
 N_p \geq \sum_{i=1}^s[N(p_if)- \theta_{p_i}N(f)]+\delta N(f).
 \end{equation}

 \end{lem}

 \begin{lem}\label{brown}
 Let $f$ be the core of $p-1$ and let $p_i$ be any prime dividing $p-1$ but not $f$ (as before). Then

\[ |N(p_if)- \theta_{p_i}N(f)|<2\left(1-\frac{1}{p_i}\right)W_fT_p^2 \ p^{\frac{1}{2}}\log^2 p.\]

\end{lem}

\begin{proof}
We have $D=N(p_if)- \theta_{p_i}N(f)=\frac{1}{2}(E(p_i f)-\theta_{p_i}E(f))$, where $E(e)$ is defined in (\ref{grey}).
Since $\theta_{p_i f}=\theta_{p_i} \theta_f=\left(1-\frac{1}{p_i}\right)\theta_f$ then, as in (\ref{black}),
\begin{equation}\label{mauve}
|D|=\frac{\theta(p_if)}{p^2}\sum_{d|f}\frac{|\mu(p_id)|}{\phi(p_id)}\sum_{\chi_{p_id}}\sum_{j,k=1}^{p-1}\big{|}\sum_{a=1}^{p-1}\chi_d(a)\psi(ja+k\bar{a})\big{|} \left|\sum_{r=1}^{p-1}(-1)^r\psi(-jr)\right|
\left|\sum_{s=1}^{p-1}(-1)^s \psi(-ks)\right|.
\end{equation}
The result follows from (\ref{mauve}) as the  deduction of (\ref{green}) from (\ref{blue}) and (\ref{Tp}).
\end{proof}

\begin{thm} \label{silver}  Let $p (>3)$ be an odd prime such that $p-1$ has (even) core $f$ and sieving primes $p_1,\ldots, p_s$,
Assume that $\delta>0$.   Then

\[N_p>\frac{\theta(f)}{2}\left\{(p-1)-2T_p^2W(f)\left(\frac{s-1}{\delta}+2\right)p^{\frac{1}{2}}\log ^2 p\right\}.\]
Hence there exists an LPR modulo $p$ whenever
\begin{equation} \label{gold}
p^{\frac{1}{2}}>2T_p^2W(f)\left(\frac{s-1}{\delta}+2\right)\log^2p + p^{-\frac{1}{2}}.
\end{equation}
For example, if $p \geq 1637$, then it suffices that
\begin{equation} \label{purple}
p^{\frac{1}{2}}> W(f)\left(\frac{s-1}{\delta}+2\right)\log^2p + p^{-\frac{1}{2}}.
\end{equation}
\end{thm}
\begin{proof}
Inequality (\ref{gold}) follows from (\ref{old}) using Lemma \ref{brown} and (\ref{Neeq}). For (\ref{purple}), recall Lemma~\ref{red}.
\end{proof}
Theorem \ref{silver} extends Theorem \ref{Np} and allows us to proceed to a complete existence result. We begin with the Corollary \ref{exist}. We use a result of Robin \cite[Thm 1]{Robin}, namely that $\omega(n) \leq 1.4 \log n/ (\log\log n)$ for all $n\geq 3$. Sharper versions of this inequality are known, but this is sufficient to show that (\ref{dundee}) holds, and thus  there is an LPR mod $p$, for all $\omega(p-1) \geq 13$. 

Next, we use (\ref{purple}) in Theorem \ref{silver} to eliminate $\omega(p-1) = 12$ by choosing $s=3$. We have  $\delta \geq 1 - 1/29 - 1/31- 1/37$ so that (\ref{purple}) is true for all $p> 3.2\cdot 10^{12}$. But, since $\omega(p-1) = 12$ we have $p-1\geq p_{1} \cdots p_{12} > 7\cdot 10^{12}$, whence we are done. Similarly, we choose $s=5,6$ for $\omega(p-1) = 11,10$.

When $\omega(p-1) = 9$ we choose $s=7$, which means that (\ref{purple}) is true for all $p\geq 1.3\cdot 10^{9}$. However, since we only know that $p-1 \geq p_{1} \cdots p_{9} > 2.2\cdot 10^{8}$ we still have some cases to check.  We proceed according to the `divide and conquer' scheme of \cite{MTT}.

We have that $3|p-1$ since otherwise $p-1 \geq 2\cdot 5 \cdots p_{10} > 2.1\cdot 10^{9}$. Moreover, we  have that $5$ divides $p-1$, since, if not, then our value of $\delta$ increases by $1/5 - 1/\prime{10}$, which is enough to show that (\ref{purple}) holds. A similar conclusion holds with the case $7|(p-1)$. While we cannot deduce that $11|(p-1)$ using this method, this is more than sufficient for our needs. We have that $p-1 = 2\cdot 3 \cdot 5 \cdot 7 k = 210k$ where, since $p< 1.3\cdot 10^{9}$ we have $k\leq 6.2\cdot 10^{6}$. 

We now enumerate all values of $n=210k+1$ for $1\leq k \leq 6.2\cdot 10^{6}$, and test whether these $n$ are prime and whether $\omega(n-1) =9$. We are left with a list of 81 values, which we can test\footnote{We could proceed, as in \cite{COT15} and \cite{COT}, to compute the \textit{exact} value of $\delta$ for these values. For example, the largest element in our list is 1,295,163,870: when $s=7$ this gives $\delta = 0.39\ldots$, which is an improvement on the worst-case scenario of $\delta = 0.33\ldots$. We find that all but 39 values in our list satisfy (\ref{purple}).} this directly to see whether they have an LPR: all do.

For $\omega(p-1) = 8, 7$ we choose $s=6,5$ which shows that we need only check those $p\leq 6.3\cdot 10^{8}$ and $p\leq 3.1\cdot 10^{8}$ respectively. For $\omega(p-1)\leq 6$ we use the unsieved (\ref{dundee}) to show that we need only check $p\leq 7.1\cdot 10^{8}$. While we could refine each of these searches, we shall simply verify that each of the 36,743,905 primes not exceeding $7.1\cdot 10^{8}$ have an LPR. 

We simply search for the first positive primitive root mod $p$, and test whether the sum of it and its inverse is odd. Once we have verified this for one value of $p$ we move on to the next one. It took less than an hour on a standard desktop (3.4 GHz Intel\textsuperscript{\textregistered} Core\texttrademark i7-6700).

\section{The Golomb pairs problem}\label{Macauley}

The following application of the theorem of Castro and Moreno (see \cite{CP}), is an instant improvement of Lemma 2.3 of \cite{WW}.
\begin{lem}\label{orange}
Let $p>3$ be prime and  $\psi$ be the additive character on the integers modulo $p$.    Further let $\chi^{(1)}, \chi^{(2)}$ be multiplicative characters modulo $p$.
Then for integers $j,k$ with $1 \leq j,k\leq p-1$,
\[ \left|\sum_{a=1}^{p-1} \chi^{(1)}(a)\chi^{(2)}(1-a) \psi(ja+k\bar{a})\right|\leq 3p^{\frac{1}{2}}. \]

\end{lem}

From now on abbreviate $\theta_{p-1}$ to $\theta$ and $W_{p-1}$ to $W$. We allow the consideration of arbitrary integers modulo $p$ but continue to restrict $\bar{a}$ for $a$ indivisible by $p$
 to mean its inverse in the range $1 \leq \bar{a} \leq  p-1$. In particular, if $a \equiv a' \pmod p$, then $\chi(a)= \chi(a')$ and $\bar{a}=\bar{a'}$.

 Drawing on \cite[\S 3]{WW}  we have

\[G_p=\frac{1}{4}\theta^2 \sum_{d_1,d_2|p-1} \frac{\mu(d_1)\mu(d_2)}{\phi(d_1)\phi(d_2)} \sum_{\chi_{d_1},\chi_{d_2}}\sum_{a=1}^{p-1}\chi_{d_1}(a)\chi_{d_2}(1-a)(1-(-1)^{a + \bar{a}}) (1-(-1)^{p+1-a+\overline{p+1-a}}).\]
Here, the sum over $a$ can omit $a=1$ because of the factor  $\chi_{d_2}(1-a)$.
Thus, $G_p=A_1-A_2-A_3+A_4$, where, for $i=1,\ldots,4$,
\begin{equation}\label{theAs}
A_i=\frac{1}{4}\theta^2 \sum_{d_1,d_2|p-1} \frac{\mu(d_1)\mu(d_2)}{\phi(d_1)\phi(d_2)}\sum_{a=2}^{p-1} \sum_{\chi_{d_1},\chi_{d_2}}\chi_{d_1}(a)\chi_{d_2}(1-a)\alpha_i,
\end{equation}
with $\alpha_1=1, \alpha_2=(-1)^{a+\bar{a}}, \alpha_3=(-1)^{a+\overline{p+1-a}}, \alpha_4=(-1)^{\bar{a}+ \overline{p+1-a}}$.  In fact, as noted in the proof in  \cite[\S 3]{WW},
$\overline{p+1-a}=p-\overline{a-1}$, so that $\alpha_4=-(-1)^{\bar{a}-\overline{a-1}}=-(-1)^{\bar{a}+\overline{a-1}}.$
Now $4A_1$ is just  the total number of pairs   $(a,b)$ of primitive roots  (not necessarily Lehmer numbers for which  $a+b \equiv 1 \pmod p$).  Hence (see, for example \cite[Lem.\ 2]{COT15}),
\begin{equation}\label{taupe}
 \left|A_1-\frac{\theta_{p-1}^2 (p-2)}{4}\right|\leq \frac{\theta_{p-1}^2}{4}(W^2-1)p^{\frac{1}{2}}.
 \end{equation}

Next, as at (\ref{black}),
\[| A_2| =\frac{\theta^2}{4p^2}\sum_{d_1,d_2|p-1}\frac{|\mu(d_1)\mu(d_2)|}{\phi_{d_1}\phi_{d_2}}\sum_{\chi_{d_1}, \chi_{d_2}} \sum_{j,k=1}^{p-1}\big{|}\sum_{a=2}^{p-1}\chi_{d_1}(a)\chi_{d_2}(1-a)\psi(ja+k\bar{a})\big{|} |U_j||U_k|, \]
where $U_j=\left|\sum_{r=1}^{p-1}(-1)^r\psi(-jr)\right|$. It makes no difference if, here, the sum over $a$ starts at $1$.
 Hence, using Lemma \ref{orange} (with the $+$ sign) instead of (\ref{blue}), the following bound holds when $i=2$, namely
\begin{equation}\label{lilac}
|A_i| \leq \frac{3\theta^2}{4} W_{p-1}^2 p^{\frac{1}{2}}\log^2 p.
\end{equation}

We demonstrate that (\ref{lilac}) also holds when $i=3,4$.   First, consider $A_3$.  Observe $\alpha_3=-(-1)^{1-a+\overline{1-a}}$.
 Replace $a$ (which runs between $2$ and $p-1$) by $p+1-a$ (which also runs between $2$ and $p-1$ and  $\alpha_3 =-(-1)^{p+a+\overline{p+a}}=(-1)^{a+\bar{a}}$.
   Moreover, in (\ref{theAs}), with $j=3$, $\chi_{d_1}(a)\chi_{d_2}(1-a)$ is transformed into $\chi_{d_2}(a)\chi_{d_1}(1-a)$.   Then, as for $A_2$,
  \[| A_3| =\frac{\theta^2}{4p^2}\sum_{d_1,d_2|p-1}\frac{|\mu(d_1)\mu(d_2)|}{\phi_{d_1}\phi_{d_2}}\sum_{\chi_{d_1}, \chi_{d_2}}
  \sum_{j,k=1}^{p-1}\big{|}\sum_{a=2}^{p-1}\chi_{d_2}(a)\chi_{d_1}(1-a)\psi(ja+k\bar{a})\big{|} |U_j||U_k|, \]
  and (\ref{lilac}) also holds when $i=3$.

  Finally, consider $A_4$.  First, set $b=a+1$ so that $b$ runs between $1$ and $p-2$ and $\alpha_4=(-1)^{b+1+\bar{b}}$.  Then set $c=\bar{b}$ (whence
  $b=\bar{c}$) so that $c$ also runs from $1$ to $p-2$ (because evidently $\overline{p-1} =p-1)$. Moreover,
  \[\alpha_4=-(-1)^{c+\overline{\bar{c}+1}}=-(-1)^{c+p+1-\overline{c+1}}=(-1)^{c+1+\overline{c+1}}.\]
  Finally, set $c=a-1$ so that this last variable $a$ again runs between $2$ and $p-1$  and $\alpha_4=(-1)^{a+\bar{a}}$. We have effectively replaced the original variable $a$ by $\frac{1}{a-1}+1= \frac{a}{a-1}$.  Hence, in the expression (\ref{theAs}) for $A_4$ we have replaced
  $\chi_{d_1}(a)\chi_{d_2}(1-a)$ by $\chi_{d_1}(a/(a-1))\chi_{d_2}(-1/(a-1))=\chi_{d_1}(-1)\chi_{d_1}(a)(\chi_{d_1}\chi_{d_2})^{-1}(1-a)$.  This yields
  \[| A_4| =\frac{\theta^2}{4p^2}\sum_{d_1,d_2|p-1}\frac{|\mu(d_1)\mu(d_2)|}{\phi_{d_1}\phi_{d_2}}\sum_{\chi_{d_1}, \chi_{d_2}}
  \sum_{j,k=1}^{p-1}\big{|}\sum_{a=2}^{p-1}\chi_{d_1}(a)(\chi_{d_1}\chi_{d_2})^{-1}(1-a)\psi(ja+k\bar{a})\big{|} |U_j||U_k|. \]
  We conclude that (\ref{lilac}) holds also when $i=4$.

  By combining (\ref{taupe}) and (\ref{lilac}) with Lemma \ref{red} we obtain a final theorem that justifies (\ref{indigo}). The inequality (\ref{kipling}) follows from (\ref{rushdie}) after a simple calculation.

  \begin{thm}\label{cream}
  Let $p>3$ be a prime.  Then
\begin{equation}\label{rushdie}
\left|G_p- \frac{\theta_{p-1}^2}{4}(p-2)\right|<\frac{\theta_{p-1}^2}{4}T_p^2[W_{p-1}^2(9\log^2 p+1) -1]p^{\frac{1}{2}}.
\end{equation}
In particular, if $p >3$, then
\begin{equation}\label{kipling}
\left|G_p- \frac{\theta_{p-1}^2}{4}(p-2)\right| <\frac{\theta_{p-1}^2}{8}[W_{p-1}^2(9\log^2 p+1) -1]p^{\frac{1}{2}}.
\end{equation}
  \end{thm}

\end{document}